\documentclass[11pt]{article} 
  \usepackage{a4wide}
  \usepackage{amsmath}
  \usepackage{amssymb}
  \usepackage{latexsym}
  \usepackage[pdftex]{graphicx}
  \usepackage{color}
  \usepackage[mathscr]{euscript}
  \usepackage[section] {placeins}

\newtheorem{theorem}{Theorem}[section]
\newtheorem{lemma}[theorem]{Lemma}
\newtheorem{definition}[theorem]{Definition}

\newtheorem{corollary}[theorem]{Corollary}

\newtheorem{remark}[theorem]{Remark}

\newenvironment{proof}{\vspace{1ex}\noindent{\bf Proof:}}{\hspace*{\fill}$\blacksquare$\vspace{1ex}}

\title{On the subgraphs of percolated random geometric graphs and the associated random complexes}
 
\author{Anshui Li\thanks{The author acknowledges support from CSC-UU scholarship grant.} }

\begin{document}
\maketitle

\begin{abstract}
In this paper, we investigate the induced subgraphs of percolated random geometric graphs, and get some asymptotic results for the expected number of the subgraph. Moreover, we get the Poisson approximation for the counting by Stein's method. We also present some similar results for the expectation of Betti number of the associated percolated random geometric complexes.
\end{abstract}

\section{Introduction}
The idea of modeling networks using random graphs was first given by Gilbert (1961) in ~\cite{gilbert1961random} where he considered a network formed by connecting points of a Poisson point process that are sufficiently close
to each other. The model Gilbert introduced was a different one from the Erd{\H o}s-Renyi random graph models in~\cite{erdos1961evolution}~\cite{erd6s1960evolution}~\cite{erdds1959random}. In this model the vertices have some (random) geometric layout and the edges are determined by the distances between the positions of the vertices. We call graphs formed in this way \emph{random geometric graphs}.

Recently, quite a lot of work has been done on \emph{random geometric graphs}, partly due to the importance of these graph model as some theoretical models for ad hoc networks, e.g., see~\cite{hekmat2006ad}. Most of the theoretical results on random geometric graphs can be found in the monograph written by Penrose~\cite{penrose2003random}.

\emph{Random geometric graphs} model is as follows. Let $f$ be some specific probability density function on $\mathbb{R}^d$, and let $X_1, X_2,...$ be independent and identically distributed $d$-dimensional random variables with common density $f: \mathbb{R}^d\to [0,\infty)$. In the whole paper, we assume that $f$ is measurable and bounded, which also satisfies $\int_{\mathbb{R}^d}f(x)dx=1.$ Let $\mathcal{X}_n=\{X_1, X_2,..., X_n\}.$ We denote $G(\mathcal{X}_n; r_n)$ the undirected graph with vertex set $\mathcal{X}_n$ and with undirected edges connecting all those pairs $\{X_i, X_j\}$ with $\lVert X_i-X_j \rVert\le r(n),$ in which $\lVert \centerdot \rVert$ denotes the Euclidean distance.

\smallskip
The \emph {random connection model} was introduced in the context of continuum percolation by Penrose ~\cite{penrose1991continuum}. Let $g\ : \mathbb{R}^2 \to [0,1]$ be such that $g(x)=g(-x).$ The function $g$ is called the \emph{connection \ function}. For two vertices $x$,$y$, they are connected with probability $g(y-x)$. Typically, it is also assumed that $g$ only depends on the distance between $x$ and $y$, i.e., $g(y-x)=\hat{g}(\lVert y-x \rVert)$ where $\hat{g}:\mathbb{R}^+ \to [0,1]$ and $\lVert \centerdot \rVert$ denotes the Euclidean norm. The \emph {random geometric graph} is a \emph {random connection model} with $\hat{g}(x)={\bf 1}_{[0,r_n]}(x)$.

\smallskip
Very recently, Penrose~\cite{penrose2016connectivity} investigated the connectivity of  \emph{random connection model }with various classes of connection functions, which are called \emph{soft random geometric graphs}. He showed that as vertex number $n \to \infty$, the probability of full connectivity is governed by that of having no isolated vertices, itself governed by a Poisson approximation for the number of isolated vertices. He generalized this beautiful result to higher dimensions, and to a large class of connection probability function in $d=2$.

\smallskip
In this paper, we consider a specific connection function which is also mentioned in Penrose ~\cite{penrose2016connectivity}: $\hat{g}(x)=p_n{\bf 1}_{[0,r_n]}(|x|)$, for some $p_n \in [0,1]$. The \emph {soft random geometric graph} gotten by this connection function, is called \emph {percolated random geometric graph}. To be more precise, a \emph{percolated random geometric graph} is defined as a random graph with vertex set $\mathcal{X}_n=\{X_1, X_2,..., X_n\}$ in which $n$ vertices are chosen at random and independently from distribution in $\mathbb{R}^2$ with probability density $f$, and a pair of vertices with Euclidean distance $r$ appears as an edge with probability $p_n$, some function of $n$, independently for each such a pair, we denote this graph $G(\mathcal{X}_n; p_n, r_n)$. In particular, for $p=1$, we can get the \emph {classic random geometric graph}, which we denote $G(\mathcal{X}_n; r).$ \emph{Hereafter, we always consider $p=p(n)$ as a function of $n$}.

In this paper, we focus on the \emph{induced} subgraph count problem on \emph {percolated random geometric graph} $G(\mathcal{X}_n; p, r)$. Let $\Gamma$ be a fixed connected graph on $k$ vertices, $k\ge 2$. Consider the number of \emph{induced} subgraphs of $G(\mathcal{X}_n; p, r)$ isomorphic to $\Gamma$. In~\cite{penrose2003random}, the author always assumes that the subgraph $\Gamma$ is \emph{feasible}, which means
 $$Pr[(\mathcal{X}_k; r) \cong \Gamma]>0$$
 for some $r >0$. However, we will not make this assumption in this paper: the subgraphs are always \emph{feasible} for \emph{percolated random geometric graphs}. Surprisingly, we can attain the asymptotic results for the means of the $\Gamma$-subgraph counts on $G(\mathcal{X}_n; p, r)$ with the help of $\Gamma$-subgraph counts on $G(\mathcal{X}_n; r)$, given that  $\Gamma$ is a clique (i.e., a complete graph), see Corollary~\ref{thm:clique}. But when it comes to general \emph{induced} subgraph and component, we can only get a lower bound for the asymptotic result for the means of the $\Gamma$-component counts on $G(\mathcal{X}_n; p, r)$ for a wide range of $p_n$, see Theorem~\ref{thm:subgraph} and Corollary~\ref{thm:tree}. The main reason behind this is that there exist many subgraphs which are feasible for $G(\mathcal{X}_n; p, r)$ but not for   $G(\mathcal{X}_n; r)$, which makes the counting on $G(\mathcal{X}_n; p, r)$ more complicated.

\smallskip
Recent years have seen an explosive growth in research of \emph {random geometric simplicial complexes}. \emph {Random simplicial complexes} may be viewed as higher dimensional generalizations of random graphs. Simplicial complex analogues of the classic Erd{\"o}s-Renyi model and their topological properties have been the subjects of many literatures in recent years. See for example~\cite{kahle2014topology},~\cite{linial2006homological}, and~\cite{meshulam2009homological}.
It is also natural to generalize the random geometric graphs, and a lot of references can be found in the survey articles~\cite{kahle2011random}. As Kahle mentioned in~\cite{kahle2011random}, two natural ways of extending a geometric graph to a simplicial complexes are: the Cech complex and the Vietoris-Rips complex(see formal definitions in the following). Most of the results in the researches of the topology of random geometric complexes are related to their homology. Briefly speaking, if $X$ is a topology space, its degree $i$-homology, denote by $H_i(X)$ is a vector space. The dimension dim$H_0(X)$ is the number of connected components of $X$, and for $i>0$, $H_i(X)$ contains information about $i-$dimensional "holes".

Since we focus on "counting" in this paper, we will also count the expected number of "holes" for the corresponding \emph{percolated random geometric complexes} in this paper. We also get the expectation of Betti number of the percolated random geometric complex (see formal definitions below).

\smallskip
Our argument is based on "coupling" of two random graph models: $G(\mathcal{X}_n; p, r)$ and $G(\mathcal{X}_n; r)$; and then use the same technique in Chapter 3 in Penrose~\cite{penrose2003random}.

The paper is organized as follows: In Section 2 we present our main results. In Section 3 we prove the main results. Finally, we note some possible generalizations and remarks in Section 4.

\section{Main results}
\subsection{Counting on percolated random geometric graphs}
We first present one asymptotic result for the means of the $\Gamma$-subgraphs counts $G_n$ in ~\cite{penrose2003random} given by Penrose. Given a connected graph $\Gamma$ on $k$ vertices, and given $A\subseteq \mathbb{R}^d$, define the indicator functions $h_{\Gamma}(\mathcal{Y})$, $h_{n, A, \Gamma}(\mathcal{Y})$, $\hat{h}_{\Gamma}(\mathcal{Y})$ and $\hat{h}_{n, A, \Gamma}(\mathcal{Y})$ for finite $\mathcal{Y}\subset \mathbb{R}^d$ by

$$h_{\Gamma}(\mathcal{Y}):=1_{\{G(\mathcal{Y};  r_n) \cong \Gamma\}},$$
$$h_{n, A, \Gamma}(\mathcal{Y}):=1_{\{G(\mathcal{Y};  r_n) \cong \Gamma\} \cap \{LMP(\mathcal{Y})\in A\}},$$
and 
$$\hat{h}_{\Gamma}(\mathcal{Y}):=1_{\{G(\mathcal{Y};  r_n, p) \cong \Gamma\}},$$
$$\hat{h}_{n, A, \Gamma}(\mathcal{Y}):=1_{\{G(\mathcal{Y};  r_n, p) \cong \Gamma\} \cap \{LMP(\mathcal{Y})\in A\}},$$

in which $LMP(\mathcal{Y})$ means the left-most point of set $\mathcal{Y}$. It is easy to observe that $h_{\Gamma}(\mathcal{Y}) =h_{n, A, \Gamma}(\mathcal{Y})=\hat{h}_{\Gamma}(\mathcal{Y})=\hat{h}_{n, A, \Gamma}(\mathcal{Y})=0$ unless $\mathcal{Y}$ has $k$ elements. 

Similarly, we define
$$g_{\Gamma}(\mathcal{Y}):=1_{\{G(\mathcal{Y};  r_n) \supsetneqq \Gamma\}},$$
and 
$$g_{n, A, \Gamma}(\mathcal{Y}):=1_{\{G(\mathcal{Y};  r_n) \supsetneqq \Gamma\} \cap \{LMP(\mathcal{Y})\in A\}},$$
in which ${\{G(\mathcal{Y};  r_n) \supsetneqq \Gamma\}}$ means $\Gamma$ is a subgraph of $G(\mathcal{Y};  r_n)$, but not equals $G(\mathcal{Y};  r_n)$. {\bf Hereafter we define $g_{\Gamma}(\mathcal{Y}) =g_{n, A, \Gamma}(\mathcal{Y})=0$ unless $\mathcal{Y}$ has $k$ elements, which means we just need to consider the graph $G(\mathcal{Y};  r_n)$ with order $k$}.

\smallskip

The reader should keep in mind that all the functions $h_{\Gamma}(\cdot)$, $h_{n,A,\Gamma}(\cdot)$, $g_{\Gamma}(\cdot)$, $g_{n,A,\Gamma}(\cdot)$ are defined on random geometric graph $G(\mathcal{X}_n; r)$; and only functions $\hat{h}_{\Gamma}(\cdot)$, $\hat{h}_{n,A,\Gamma}(\cdot)$, are defined on percolated random geometric graph $G(\mathcal{X}_n; r, p)$.

\smallskip
We set
$$\mu_{\Gamma, A}:= k!^{-1}\int_{A}f(x)^kdx\int_{(\mathbb{R}^d)^{k-1}}h_{\Gamma}(\{0, x_1,...,x_{k-1}\})d(x_1,...,x_{k-1}),$$
$$\hat{\mu}_{\Gamma, A}:= k!^{-1}\int_{A}f(x)^kdx\int_{(\mathbb{R}^d)^{k-1}}\hat{h}_{\Gamma}(\{0, x_1,...,x_{k-1}\})d(x_1,...,x_{k-1}),$$
$$\mu'_{\Gamma, A}:= k!^{-1}\int_{A}f(x)^kdx\int_{(\mathbb{R}^d)^{k-1}}g_{\Gamma}(\{0, x_1,...,x_{k-1}\})d(x_1,...,x_{k-1}).$$
We write $\mu_{\Gamma}$, $\hat{\mu}_{\Gamma}$, $\mu'_{\Gamma}$ for $\mu_{\Gamma, \mathbb{R}^d}$, $\hat{\mu}_{\Gamma, \mathbb{R}^d}$ and $\mu'_{\Gamma, \mathbb{R}^d}$ respectively.

\smallskip
Let $G_{n,A}(\Gamma)$ and $G'_{n,A}(\Gamma)$ be the number of \emph {induced} subgraphs of $G(\mathcal{X}_n; r)$ and $G(\mathcal{X}_n; p,r)$ for which the left-most of the vertex set lies in $A$, respectively.

\begin{theorem}[Penrose~\cite{penrose2003random}]\label{thm:penrosesubgraph}
Suppose that $\Gamma$ is a feasible connected graph of order $k\geq 2$, that $A \subseteq \mathbb{R}^d $ is open with $Leb(\partial A)=0$, and that $\lim_{n\to \infty}(r_n)=0.$ Then
$$\lim_{n\to \infty}r_n^{-d(k-1)}n^{-k}E(G_{n,A}(\Gamma))=\mu_{\Gamma,A}.$$
\end{theorem}

Similar with the result above, we count the \emph{induced} subgraph in the percolated random geometric graph $G(\mathcal{X}_n; p,r)$ and get one theorem below:

\begin{theorem}\label{thm:similar}
Suppose that $\Gamma$ is a connected graph of order $k\geq 2$, that $A \subseteq \mathbb{R}^d $ is open with $Leb(\partial A)=0$, and that $\lim_{n\to \infty}(r_n)=0.$ Then
$$\lim_{n\to \infty}r_n^{-d(k-1)}n^{-k}E(G'_{n,A}(\Gamma))=\hat{\mu}_{\Gamma,A}.$$
\end{theorem}

\smallskip
However, if have more information about graph $\Gamma$, we can get more detailed results. In the following, we will present some results related to induced-graph $\Gamma$ with order $k\ge 2$ and size $m\ge 1$.

\begin{theorem}[Counting of induced subgraph]\label{thm:subgraph}
Suppose that $\Gamma$ is a connected graph of order $k\geq 2$ and size $m$, that $A \subseteq \mathbb{R}^d $ is open with $Leb(\partial A)=0$, and that $\lim_{n\to \infty}(r_n)=0.$ Then\\
if $p_n \equiv p$, we have
$$\lim_{n\to \infty}p^{-m}n^{-k}r_n^{-d(k-1)} E(G'_{n,A}(\Gamma)) \geq \mu_{\Gamma, A} +(1-p)^{\binom{k}{2}}\mu'_{\Gamma, A}.$$
If $\lim_{n\to\infty}n^2p_n\to \alpha\in (0,\infty)$, we have
$$\lim_{n\to \infty}p^{-m}n^{-k}r_n^{-d(k-1)} E(G'_{n,A}(\Gamma)) \geq \mu_{\Gamma, A} +e^{-\alpha/2}\mu'_{\Gamma, A}.$$
If $\lim_{n\to\infty}n^2p_n\to 0$, we have
$$\lim_{n\to \infty}p^{-m}n^{-k}r_n^{-d(k-1)} E(G'_{n,A}(\Gamma)) \geq \mu_{\Gamma, A} +\mu'_{\Gamma, A}.$$
\end{theorem}
\begin{corollary}[Counting of tree-subgraph ]\label{thm:tree}
Suppose that $\Gamma$ is a connected graph of order $k\geq 2$ and size $m = k-1$, that $A \subseteq \mathbb{R}^d $ is open with $Leb(\partial A)=0$, and that $\lim_{n\to \infty}(r_n)=0.$ Then \\
if $p_n\equiv p$, we have
$$  \lim_{n\to \infty}\frac{E(G'_{n,A}(\Gamma))}{n}\left(\frac{\theta}{d(n)}\right)^{k-1}\ge \mu_{\Gamma,A}+(1-p)^{\binom{k}{2}-(k-1)}\mu'_{\Gamma,A};$$
If $n^2p_n\to \alpha\in (0,\infty)$, we have
$$  \lim_{n\to \infty}\frac{E(G'_{n,A}(\Gamma))}{n}\left(\frac{\theta}{d(n)}\right)^{k-1}\ge \mu_{\Gamma,A}+e^{-\alpha/2}\mu'_{\Gamma,A};$$
If $n^2p_n\to 0$, we have
$$  \lim_{n\to \infty}\frac{E(G'_{n,A}(\Gamma))}{n}\left(\frac{\theta}{d(n)}\right)^{k-1}\ge \mu_{\Gamma,A}+\mu'_{\Gamma,A},$$
in which $d(n)=n\theta r_n^d p_n$.

\end{corollary}
\begin{corollary}[Counting of clique-subgraph]\label{thm:clique}
Suppose that $\Gamma$ is a clique of order $k\geq 2$, that $A \subseteq \mathbb{R}^d $ is open with $Leb(\partial A)=0$, and that $\lim_{n\to \infty}(r_n)=0.$ Then
$$E(G'_{n,A}(\Gamma))=p^{\binom{k}{2}}E(G_{n,A}(\Gamma)).$$
Moreover, we can get
$$\lim_{n\to \infty}p^{-\binom{k}{2}}n^{-k}r_n^{-d(k-1)}E(G'_{n,A}(\Gamma))=\mu_{\Gamma, A}.$$

\end{corollary}


\bigskip
Next consider the component count in the thermodynamic limit where $n r_n^d$ tends to a constant. Given $\lambda >0$, and given a feasible connected graph $\Gamma$ of order $k\geq 2$, define
$$p_{\Gamma}(\lambda):=\frac{\lambda^{k-1}}{(k-1)!}\int_{(\mathbb{R}^d)^{k-1}}h_{\Gamma}(\{0,x_1,...,x_{k-1}\}) \times \exp(-\lambda V(0,x_1,...,x_{k-1}))d(x_1,...,x_{k-1})$$
and
$$\hat{p}_{\Gamma}(\lambda):=\frac{\lambda^{k-1}}{(k-1)!}\int_{(\mathbb{R}^d)^{k-1}}\hat{h}_{\Gamma}(\{0,x_1,...,x_{k-1}\}) \times \exp(-\lambda V(0,x_1,...,x_{k-1}))d(x_1,...,x_{k-1}),$$
and
$$p'_{\Gamma}(\lambda):=\frac{\lambda^{k-1}}{(k-1)!}\int_{(\mathbb{R}^d)^{k-1}}g_{\Gamma}(\{0,x_1,...,x_{k-1}\}) \times \exp(-\lambda V(0,x_1,...,x_{k-1}))d(x_1,...,x_{k-1}),$$
where $V(y_1,...,y_m)$ denotes the Lebesgue measure (volume) of the union of balls of unit radius (in the chosen norm) centered at $y_1, ..., y_m$. If $\Gamma$ consists of one single point (i.e. if $k=1$), set $p_{\Gamma}(\lambda)=p'_{\Gamma}(\lambda)=\hat{p}_{\Gamma}(\lambda):=\exp(-\lambda \theta)$, in which $\theta$ is the volume of the unit ball in $\mathbb{R}^d$.

Let $J_{n,A}(\Gamma)$ be the number of $\Gamma$-components of $G(\mathcal{X}_n; r)$ for which the left-most point of the vertex set lies in $A$; and $J'_{n,A}(\Gamma)$ be the number of $\Gamma$-components of $G(\mathcal{X}_n; p, r)$ for which the left-most point of the vertex set lies in $A$.

\begin{theorem}[Penrose~\cite{penrose2003random}]\label{thm:penrosecomponent}
Suppose that $A\subseteq {\mathbb R}^d$ is open with $Leb({\partial A})=0$, that $\Gamma$ is a feasible connected graph order $k\in \mathbb{N}$, and that $n r_n^d\to \rho \in (0,\infty)$. Then
$$\lim_{n\to \infty}\left(\frac{E(J_{n, A}(\Gamma))}{n}\right)= k^{-1}\int_A p_{\Gamma}(\rho f(x))f(x)dx.$$
\end{theorem}

For the percolated random geometric graphs, we have one similar result as follows.

\begin{theorem}\label{thm:similar2}
Suppose that $A\subseteq {\mathbb R}^d$ is open with $Leb({\partial A})=0$, that $\Gamma$ is a connected graph order $k\in \mathbb{N}$, and that $n r_n^d\to \rho \in (0,\infty)$. Then
$$\lim_{n\to \infty}\left(\frac{E(J_{n, A}(\Gamma))}{n}\right)= k^{-1}\int_A \hat{p}_{\Gamma}(\rho f(x))f(x)dx.$$
\end{theorem}

Same story as the counting of \emph{induced} subgraph, we can get more detailed results if have more information about the \emph{induced} component.
\begin{theorem}[Counting of $\Gamma$-component]\label{thm:component}
Suppose that $A\subseteq {\mathbb R}^d$ is open with $Leb({\partial A})=0$, that $\Gamma$ is a connected graph order $k\in \mathbb{N}$ and size $m$, and that $n r_n^d\to \rho \in (0,\infty)$. Then \\
if $p_n\equiv p$, we have
$$ \lim_{n\to \infty}\left(\frac{E(J'_{n, A}(\Gamma))}{n p^m_n}\right)\geq k^{-1}\int_A p_{\Gamma}(\rho  f(x))f(x)dx+k^{-1}(1-p)^{\binom{k}{2}-m}\int_{A}p'_{\Gamma}(\rho  f(x))f(x)dx.$$
If $n^2p_n\to \alpha\in (0,\infty)$, we have
$$ \lim_{n\to \infty}\left(\frac{E(J'_{n, A}(\Gamma))}{n p_n^m}\right)\geq k^{-1}\int_A p_{\Gamma}(\rho  f(x))f(x)dx+k^{-1}e^{-\alpha/2}\int_{A}p'_{\Gamma}(\rho  f(x))f(x)dx.$$
If $n^2p_n\to 0$, we have
$$ \lim_{n\to \infty}\left(\frac{E(J'_{n, A}(\Gamma))}{n p^m_n}\right)\geq k^{-1}\int_A p_{\Gamma}(\rho  f(x))f(x)dx+k^{-1}\int_{A}p'_{\Gamma}(\rho  f(x))f(x)dx.$$

\end{theorem}

\begin{corollary}[Counting of clique-component]\label{thm:clique-component}
Suppose that $A\subseteq {\mathbb R}^d$ is open with $Leb({\partial A})=0$, that $\Gamma$ is a clique with order $k\in \mathbb{N}$, and that $n r_n^d\to \rho \in (0,\infty)$. Then we have

$$E(J'_{n, A}(\Gamma))\geq p_n^{\binom{k}{2}} E(J_{n, A}(\Gamma)).$$
Moreover, we have 
\begin{equation}\label{equ:6}
 \lim_{n\to \infty}\left(\frac{E(J'_{n, A}(\Gamma))}{n p_n^{\binom{k}{2}}}\right)\geq k^{-1}\int_A p_{\Gamma}(\rho  f(x))f(x)dx.
 \end{equation}
\end{corollary}

In the following subsection, we will present the basic Poisson approximation theorem for the induced $\Gamma$-subgraph count $G'_n$ on \emph{percolated random geometric graph} $G(\mathcal{X}_n; p, r).$
Compare to the similar results for random geometric graphs in ~\cite{penrose2003random}, the \emph{total variation distance} between the distribution of $G'_n$ and corresponding Poisson distribution is tighter for percolated random geometric graphs.

\begin{theorem}\label{thm:poisson}
Let $\Gamma$ be a connected graph of order $k\ge 2$ and size $m$, and we define $G'_n:=G'_{n,\mathbb{R}^d}(\Gamma)$. Suppose $(nr_n^d)_{n\ge 1}$ is a bounded sequence. Let $Z_n$ be Poisson with parameter $E(G'_n).$
Then there exists a constant $c$ such that for all n,
\begin{displaymath}\label{equ:poissonappro}
d_{TV}(G'_n, Z_n)\le \left\{\begin{array}{ll}  cnp_n^{2m+2-k}r^d_n  &\textrm{if $k\ge 4$}\\
cnp_n^{2m-2}r_n^d &\textrm{if $2\le k<4$}
\end{array}\right.
\end{displaymath}

If $n^kr^{d(k-1)}_n \to \alpha \in(0,\infty)$, then $G'_n \xrightarrow[] {D} Po(\lambda)$ with $\lambda=\alpha \hat{\mu}_{\Gamma}$.

If $n^kr^{d(k-1)}_n \to \infty$ and $nr^d_n \to 0$, then 
$\left(n^kr^{d(k-1)}_n\hat{\mu}_{\Gamma} \right)^{-1/2}(G'_n-EG'_n)\xrightarrow[]{D}\mathcal{N}(0,1).$

\end{theorem}

\subsection{Counting on random geometric complexes}
In this section, we present some very preliminary results related to the \emph{percolated random geometric complexes}, which are the corresponding results in \emph{percolated} version for the expectation of Betti numbers of Vietoris-Rips complex as in Kahle~\cite{kahle2011random}.

For completeness of this paper, we first review some definitions related to simplicial complexes. A set of $k+1$ points, $u_0, u_1,..., u_k$, is \emph{affinely independent} if the $k$ vectors, $u_1-U-0, u_2-u_0,...,u_k-u_0$, are linear independent. A \emph{k-simplex} is the convex hull of $k+1$ affinely independent points. Writing $\sigma$ for the $k$-simplex, we call $k=$dim$\sigma$ its dimension, and $u_0$ to $u_k$ its \emph{vertices}. Simplifies of dimension $0, 1, 2, 3$ are usually referred to as \emph{vertices, edges, triangles, tetrahedra}. A \emph{face} of $\sigma$ is a simplex spanned by a subset of the vertices of $\sigma$. Since a set of $k+1$ elements has $\binom{k+1}{l+1}$ subsets of size $l+1$, $\sigma$ has this number of $l$-faces, for $0\le l\le k$. The total number of faces is 
$$\sum_{l=0}^k\binom{k+1}{l+1}=2^{k+1}-1,$$
the number of subsets minus 1 means we do not count the empty set. We then define a \emph{simplicial complex} as a finite collection of simplices, $K$, such that
\begin{enumerate}
\item for every simplex $\sigma\in K$, every face of $\sigma$ is in $K$;
\item for every two simplifies $\sigma, \tau \in K$, the intersection, $\sigma\cap \tau$, is either empty or a face of both simplices.
\end{enumerate}
If the intersection of two simplifies is a common face, then $(i)$ implies that it is a simplex in $K$. The \emph{dimension} of a simplicial complex $K$ is the largest dimension of any simplex in $K$. A \emph{sub complex} of $K$ is the simplices that is itself a simplicial complex. For more detains on simplicial complex and related properties, we recommend the brief monograph~\cite{edelsbrunner2014short} by Edelsbrunner.

\smallskip
The random geometric complexes studied are simplicial complexes built on independent and identically distributed random points in Euclidean space $\mathbb{R}^d$. In this section, we make mild assumptions about the common density $f$: $f$ is bounded Lebesgue-measurable function and 
$$\int_{\mathbb{R}^d}fdx=1.$$

The main object of study in this section is the \emph{percolated} Vietoris-Rips complexes on ${X_1, X_2,...,X_n}$, which is a sequence of independent and identically distributed $d$-dimensional random variables with common density $f$, we denote the sequence by $\mathcal{X}_n=\{X_1, X_2,...,X_n\}$. The Vietoris-Rips complex was first introduced by Vietoris in order to extend simplicial homology to a homology theory for metric spaces~\cite{vietoris1927hoheren}. Eliyahu Rips applied the same complex to the study of hyperbolic groups, and Gromov popularized the name of Rips complex~\cite{gromov1987hyperbolic}.

Denote the closed ball of radius $r$ centered at a point $p$ by $B(p,r)=\{x\mid ||x-p||\le r\}$, in which $||\cdot||$ is the Euclidean distance in $\mathbb{R}^d$.

The formal definition of Vietoris-Rips complex goes as follows:
\begin{definition}[Random VR complex]
The random Vietoris-Rips complex $R(\mathcal{X}_n; r)$ is the simplicial complex with vertex set $\mathcal{X}_n$ and $\sigma$ a face if
$$B(X_i,r/2)\cap B(X_j,r/2)\not= \emptyset$$
for every pair $X_i,\ X_j\in \sigma$.
\end{definition}
From the definition above, it is easy to see that the random Vietoris-Rips complex is the clique complex of $G(\mathcal{X}_n; r)$.

As we mentioned before, we want to study one \emph{percolated} version of the random Vietoris-Rips complex. Roughly speaking, the underlying graph for the random Vietoris-Rips complex is the classic random graphs $G(\mathcal{X}_n;r)$, while the underlying graph for the \emph{percolated} random Vietoris-Rips complex is the \emph{percolated} random geometric graph $G(\mathcal{X}_n; r,p)$. 

\begin{definition}[Percolated random VR complex]
Let $G(\mathcal{X}_n; r,p)$ be the percolated random geometric graph built on the random points $\mathcal{X}_n=\{X_1,X_2,...,X_n\}$, which are $i.i.d$ with common density $f$. The percolated random Vietoris-Rips complex $R(\mathcal{X}_n;r,p)$ associated with graph $G(\mathcal{X}_n; r,p)$ is the simplicial complex with vertex $\mathcal{X}_n$ and $\sigma$ a face if
$$(X_i, X_j)\in E\left(G(\mathcal{X}_n; r,p)\right)$$
for every pair $X_i,\ X_j\in \sigma$.
\end{definition}

In other words, we build any $k-$simplex by its basic $2$-faces, i.e., edges. A face $\sigma$ exists if all its $2$-subfaces exist.

\smallskip
 In this paper, we only mention the similar result for the expectation of Betti number in the subcritical regime.
\begin{theorem}[Betti number of random geometric VR complex,\cite{kahle2011random}]\label{thm:kahle}
For $d\ge 2$, $k\ge 1$, and $r_n=o(n^{-1/d}),$ the expectation of the $k$th Betti number $E[\beta_k]$ of the random Vietoris-Rips complex $R(\mathcal{X}_n;r)$ satisfies
$$
\frac{E[\beta_k]}{n^{2k+2}r^{d(2k+1)}}\to C_k$$
as $n\to \infty$, where $C_k$ is a constant that depends only on $k$ and the underlying density $f$.
\end{theorem}
For the percolated random Vietoris-Rips complex, we have the similar results:
\begin{theorem}[Betti number of percolated random VR complex]\label{thm:betti}
For $d\ge 2$, $k\ge 1$, and $r_n=o(n^{-1/d}),$ the expectation of the $k$th Betti number $E[\beta_k]$ of the percolated random Vietoris-Rips complex $R(\mathcal{X}_n;r,p)$ associated with graph $G(\mathcal{X}_n; r,p)$ satisfies
$$
\frac{E[\beta_k]}{n^{2k+2}r^{d(2k+1)}p^{2k(k-1)}}\to C'_k$$
as $n\to \infty$, where $C'_k$ is a constant that depends only on $k$ and the underlying density $f$.
\end{theorem}

\section{Proof of the main theorems}
\subsection{Coupling of two random geometric models: $G(\mathcal{X}_n; p, r)$ and $G(\mathcal{X}_n; r)$}
Given a vertex set $\mathcal{X}_n=\{X_1,..., X_n\}$ which are independently from a distribution on $\mathbb{R}^2$ with probability density $f$, and two functions $r_n >0, p_n\in [0,1]$. We get the  percolated random geometric graph $G(\mathcal{X}_n; p, r)$ in the following two steps:
\begin{itemize}
\item Put an edge between $X_i$ and $X_j$ if $\lVert X_i -X_j\rVert \le r$, for $1\le i<j \le n$, we get $G(\mathcal{X}_n; r)$;
\item For $G(\mathcal{X}_n; r)$ obtained above, we keep every edge with probability $p$ (i.e., we delete it with probability $1-p$), independently with all other edges. Then we get $G(\mathcal{X}_n; p, r)$.
\end{itemize}

From the procedure above, we can get that there are at least two ways to get the \emph{induced} subgraphs in $G(\mathcal{X}_n; p, r)$: 
\begin{itemize}
\item If $G(\mathcal{X}_k; r)\cong \Gamma$, we can keep it in the second step;
\item If $G(\mathcal{X}_k; r) \supsetneqq \Gamma$, we can delete the unwanted edges in the second step, and get $G(\mathcal{X}_k; r, p)\cong \Gamma$.
\end{itemize}

In short, all the \emph{induced} subgraphs $G(\mathcal{X}_k; r, p)\cong \Gamma$ are born from some graphs $G(\mathcal{X}_k; r)$ with more (or same) edges.

\smallskip
It is easy to observe that: if $(X_i, X_j) $ is one edge in $G(\mathcal{X}_n; r)$, then $(X_i, X_j)$ is one edge in $G(\mathcal{X}_n; p, r)$ with probability $p$. As a consequence, we can get a lemma.

\begin{lemma}[Coupling Lemma]~\label{lem:cou}Suppose that $\Gamma$ be a fixed connected graph of order $k \geq 2$ and size $m\ge 1$. Then  
$$Pr[G(\mathcal{X}_k; p, r)\cong \Gamma]\ge p^m Pr[G(\mathcal{X}_k; r)\cong \Gamma].$$

\end{lemma}
\begin{proof}
If $\Gamma$ is not feasible for $G(\mathcal{X}_n; r)$, i.e.,  $Pr[G(\mathcal{X}_k; r)\cong \Gamma]=0$, the statement of course holds; If $G(\mathcal{X}_k; r)\cong \Gamma$, then keep all the edges in $G(\mathcal{X}_k; r)$, we can get $G(\mathcal{X}_k; p, r)\cong \Gamma$. We complete the proof.
\end{proof}
\begin{remark}
From the Lemma~\ref{lem:cou}, every $\Gamma$-subgraph with size $m$ in $G(\mathcal{X}_n; r)$, can contribute $p_n^m$ to the expectation of number of $\Gamma$-subgraph in $G(\mathcal{X}_n; p, r)$.
\end{remark}
\subsection{Proof of Theorem~\ref{thm:similar}}
Slightly modify the proof of Theorem~\ref{thm:penrosesubgraph} in ~\cite{penrose2003random}, we can get the theorem.

\subsection{Proof of Theorem~\ref{thm:subgraph}}\label{pro:subgraph} It is easy to get
\begin{equation}\label{eq:1}
\begin{array}{rcl}
\displaystyle
Pr(G(\mathcal{X}_k; r,p)\cong \Gamma)
\displaystyle
&=&
Pr\left( G(\mathcal{X}_k; r,p)\cong \Gamma \mid G(\mathcal{X}_k; r)\cong \Gamma \right)Pr(G(\mathcal{X}_k; r)\cong \Gamma)\\
\displaystyle
&+&
Pr\left( G(\mathcal{X}_k; r,p)\cong \Gamma \mid G(\mathcal{X}_k; r)\supsetneqq \Gamma \right)Pr(G(\mathcal{X}_k; r)\supsetneqq \Gamma)\\
\displaystyle
&=&
p^mPr(G(\mathcal{X}_k; r)\cong \Gamma)\\
\displaystyle
&+&
Pr\left( G(\mathcal{X}_k; r,p)\cong \Gamma \mid G(\mathcal{X}_k; r)\supsetneqq \Gamma \right)Pr(G(\mathcal{X}_k; r)\supsetneqq \Gamma)\\
\displaystyle
&\ge&
p^mPr(G(\mathcal{X}_k; r)\cong \Gamma)+p^m(1-p)^{\binom{k}{2}-m}Pr(G(\mathcal{X}_k; r)\supsetneqq \Gamma)\\
\displaystyle
&\geq&
p^mPr(G(\mathcal{X}_k; r)\cong \Gamma)+p^m(1-p)^{\binom{k}{2}}Pr(G(\mathcal{X}_k; r)\supsetneqq \Gamma).
\end{array}
\end{equation}

The first equality means: the conditional probability $Pr\left( G(\mathcal{X}_k; r,p)\cong \Gamma \mid G(\mathcal{X}_k; r)\supsetneqq \Gamma \right)$ cannot be calculated easily, as it depends on both the structures $\Gamma$ and $G(\mathcal{X}_k; r)$. However, we can get a lower bound for this probability by considering $G(\mathcal{X}_k; r)$ as a complete graph, and delete all the unwanted $\binom{k}{2}-m$ edges. 

\smallskip
As $E(G'_{n,A}(\Gamma))=\binom{n}{k}Pr(G(\mathcal{X}_k; r,p)\cong \Gamma)$, we have
\begin{equation}\label{equ:4}
\begin{array}{rcl}
\displaystyle
E[G'_{n,A}(\Gamma)]
\displaystyle
&\geq&
\binom{n}{k}p^m(Pr(G(\mathcal{X}_k; r)\cong \Gamma)+\binom{n}{k}p^m(1-p)^{\binom{k}{2}}Pr(G(\mathcal{X}_k; r)\supsetneqq \Gamma).
\end{array}
\end{equation}
Follow the same idea of the proof of the Proposition 3.1 in ~\cite{penrose2003random}, we can get that the first term and second term on the right side of (~\ref{equ:4} ) is asymptotic to $n^kp^mr_n^{d(k-1)}\mu_{\Gamma, A}$ and $n^kp^m(1-p)^{\binom{k}{2}}r_n^{d(k-1)}\mu'_{\Gamma, A}$.
\begin{enumerate}
\item If $p_n\equiv p$, we can rearrange the terms, and get the result;
\item if $n^2p_n\to \alpha$, which means $(1-p)^{\binom{k}{2}}\geq(1-p)^{\binom{n}{2}}\sim e^{-p_nn^2/2}\sim e^{-\alpha/2}$ as $n\to \infty$;
\item if $n^2p_n\to 0$,which means $(1-p)^{\binom{k}{2}}\geq(1-p)^{\binom{n}{2}}\sim e^{-p_nn^2/2}\sim 1$ as $n\to \infty$.
\end{enumerate}
We complete our proof.

\subsection{Proof of Corollary~\ref{thm:tree}}
let $m=k-1$, use the same idea as proof of ~\ref{thm:subgraph}, we can get the Theorem~\ref{thm:tree}. 

\subsection{Proof of Corollary ~\ref{thm:clique}}
If $\Gamma$ is a clique with order $k$, we have
$$Pr(G(\mathcal{X}_k; r)\supsetneqq \Gamma)=0,$$ 
i.e.,
$$Pr(G(\mathcal{X}_k; r,p)\cong \Gamma)=Pr\left( G(\mathcal{X}_k; r,p)\cong \Gamma \mid G(\mathcal{X}_k; r)\cong \Gamma \right)Pr(G(\mathcal{X}_k; r)\cong \Gamma).$$
Use the same argument as proof of Theorem~\ref{thm:subgraph}, we can finish the proof here.

\subsection{Proof of Theorem~\ref{thm:similar2}}
Almost same argument as the proof of Theorem~\ref{thm:penrosecomponent} in ~\cite{penrose2003random}, we can get the theorem.
\subsection{Proof of Theorem~\ref{thm:component}}For the component counting, we have
\begin{equation}\label{eq:5}
\begin{array}{rcl}
\displaystyle
Pr(G(\mathcal{X}_k; r,p)\cong \Gamma)
\displaystyle
&=&
Pr\left( G(\mathcal{X}_k; r,p)\cong \Gamma \mid G(\mathcal{X}_k; r)\cong \Gamma \right)Pr(G(\mathcal{X}_k; r)\cong \Gamma)\\
\displaystyle
&+&
Pr\left( G(\mathcal{X}_k; r,p)\cong \Gamma \mid G^{dis}(\mathcal{X}_k; r)\supsetneqq \Gamma \right)Pr(G^{dis}(\mathcal{X}_k; r)\supsetneqq \Gamma)\\
\displaystyle
&+&
Pr\left( G(\mathcal{X}_k; r,p)\cong \Gamma \mid G^{con}(\mathcal{X}_k; r)\supsetneqq \Gamma \right)Pr(G^{con}(\mathcal{X}_k; r)\supsetneqq \Gamma)\\
\displaystyle
&=&
p^mPr(G(\mathcal{X}_k; r)\cong \Gamma)\\
\displaystyle
& + &
Pr\left( G(\mathcal{X}_k; r,p)\cong \Gamma \mid G^{dis}(\mathcal{X}_k; r)\supsetneqq \Gamma \right)Pr(G^{dis}(\mathcal{X}_k; r)\supsetneqq \Gamma)\\
\displaystyle
&+&
Pr\left( G(\mathcal{X}_k; r,p)\cong \Gamma \mid G^{con}(\mathcal{X}_k; r)\supsetneqq \Gamma \right)Pr(G^{con}(\mathcal{X}_k; r)\supsetneqq \Gamma)\\
\displaystyle
&\ge&
p^mPr(G(\mathcal{X}_k; r)\cong \Gamma)\\
\displaystyle
&+&
p^m(1-p)^{\binom{k}{2}-m}Pr(G^{dis}(\mathcal{X}_k; r)\supsetneqq \Gamma)\\
\displaystyle
&+&
Pr\left( G^{con}(\mathcal{X}_k; r,p)\cong \Gamma \mid G^{con}(\mathcal{X}_k; r)\supsetneqq \Gamma \right)Pr(G^{con}(\mathcal{X}_k; r)\supsetneqq \Gamma),
\end{array}
\end{equation}
in which, $G^{dis}(\mathcal{X}_k; r)$ means that $G(\mathcal{X}_k; r)$ does not connect with any vertices in $\mathcal{X}_n\setminus \mathcal{X}_k$; and $G^{con}(\mathcal{X}_k; r)$ means that $G(\mathcal{X}_k; r)$ does connect with some vertex in $\mathcal{X}_n\setminus \mathcal{X}_k$.

So we get
\begin{equation}\label{equ:5}
\begin{array}{rcl}
\displaystyle
E(J'_{n,A}(\Gamma))
\displaystyle
&=&
\binom{n}{k}Pr(G(\mathcal{X}_k; r,p)\cong \Gamma)\\
\displaystyle
&=&
\binom{n}{k}Pr\left( G(\mathcal{X}_k; r,p)\cong \Gamma \mid G(\mathcal{X}_k; r)\cong \Gamma \right)Pr(G(\mathcal{X}_k; r)\cong \Gamma)\\
\displaystyle
&+&
\binom{n}{k}Pr\left( G(\mathcal{X}_k; r,p)\cong \Gamma \mid G^{dis}(\mathcal{X}_k; r)\supsetneqq \Gamma \right)Pr(G^{dis}(\mathcal{X}_k; r)\supsetneqq \Gamma)\\
\displaystyle
&+&
\binom{n}{k}Pr\left( G(\mathcal{X}_k; r,p)\cong \Gamma \mid G^{con}(\mathcal{X}_k; r)\supsetneqq \Gamma \right)Pr(G^{con}(\mathcal{X}_k; r)\supsetneqq \Gamma).
\end{array}
\end{equation}
For the first term of the right side of (~\ref{equ:5}),  by Theorem ~\ref{thm:penrosecomponent}, we can know the the asymptotic result
$$n^{-1}\binom{n}{k}Pr\left( G(\mathcal{X}_k; r,p)\cong \Gamma \mid G(\mathcal{X}_k; r)\cong \Gamma \right)Pr(G(\mathcal{X}_k; r)\cong \Gamma)\to p^m k^{-1}\int_{A}p_{\Gamma}(\rho f(x))f(x)dx.$$

For the second term, we use the same argument as in the proof of Proposition 3.3 in~\cite{penrose2003random}, and get that
$$n^{-1}\binom{n}{k}Pr\left( G(\mathcal{X}_k; r,p)\cong \Gamma \mid G^{dis}(\mathcal{X}_k; r)\supsetneqq \Gamma \right)Pr(G^{dis}(\mathcal{X}_k; r)\supsetneqq \Gamma) $$
is asymptotically bounded from below by
$$ p^m(1-p)^{\binom{k}{2}-m}k^{-1}\int_{A}p'_{\Gamma}(\rho f(x))f(x)dx.$$
Then use the same arguments as the ~\ref{pro:subgraph}, we finish the proof.

\subsection{Proof of Corollary~\ref{thm:clique-component}}
By Lemma{~\ref{lem:cou}}, we can get 
\begin{equation}\label{eq:2}
\begin{array}{rcl}
\displaystyle
E[J'_{n,A}(\Gamma)]
\displaystyle
&\geq&
 p^{\binom{k}{2}}E[J_{n,A}(\Gamma)].
\end{array}
\end{equation}
By Theorem ~\ref{thm:penrosecomponent}, which gives us (~\ref{equ:6}).

\subsection{Proof of Theorem ~\ref{thm:poisson}} 
Before we prove this theorem, we present some notations related to \emph{dependency graphs} and some approximation results for sums of Bernoulli variables indexed by the vertices of a dependency graph.

Suppose $(I,E)$ is a graph with finite for countable vertex set $I$. For $i,j \in I$ write $i\sim j$ if $\{i,j\}\in E.$ For $i\in I$, let $\mathcal{N}_i$ denote the adjacency neighborhood of $i$, that is, the set $\{i\}\cup \{j\in I: j\sim i\}.$ We say that the graph $(I,\sim)$ is a \emph{dependency graph} for a collection of random variables $(\xi_i,i\in I)$ if for any two disjoint subsets $I_1, I_2$ of $I$ such that there are no edges connecting $I_1$ to $I_2$, the collection of random variables $(\xi_i,i\in I_1)$ is independent of $(\xi_j, j\in I_2).$ The notation of \emph{dependency graph} is very helpful to cope with some problem related to near-independence random variables.

\smallskip
\begin{theorem}[Arratia et al. 1989~\cite{arratia1989two}]~\label{thm:stein}
Suppose $(\xi_i, i \in I)$ is a finite collection of Bernoulli random variables with dependency graph $(I,\sim)$. Set $p_i:=E(\xi_i)=P[\xi_i=1]$, and set $p_{ij}:=E[\xi_i \xi_j].$ Let $\lambda:=\sum_{i\in I}p_i,$ and suppose $\lambda$ is finite. let $W:=\sum_{i\in I}\xi_i$. Then
$$d_{TV}(W, Po(\lambda)) \leq \min(3, \lambda^{-1})\left(\sum_{i\in I}\sum_{j\in \mathcal{N}_i\setminus \{i\}}p_{ij} +\sum_{i\in I}\sum_{j\in \mathcal{N}_i}p_i p_j\right).$$
\end{theorem}

\smallskip
\begin{proof}[Proof of Theorem~\ref{thm:poisson}]
Clearly we have
$$G'_n=\sum_{{\bf i}\in \mathcal{I}_n}\xi_{{\bf i},n},$$
where $\bf{i}$ runs through the index set $\mathcal{I}_n$ of all $k$-subsets ${\bf i}=\{i_1,..., i_k\}$ of $\{1,2,...,n\},$ and $\xi_{{\bf i}, n}=1_{\{G(\{X_i, \ i\in {\bf i}\}; p, r)\cong \Gamma\}}.$

Then we use \emph{stein's} method to get the error bounds for the convergence.

For each index ${\bf i} \in \mathcal{I}_n$, let $\mathcal{N}_i$ be the set of ${\bf j} \in \mathcal{I}_n$ such that ${\bf i}$ and ${\bf j}$ have at least one element in common. Let $\sim$ be the associated adjacency relation on $\mathcal{I}_n$, that is ${\bf i}\sim {\bf j}$ if ${\bf j} \in \mathcal{N}_i$ and ${\bf i}\neq {\bf j}$. Then $\xi_{{\bf i},n}$ is independent of $\xi_{{\bf j},n}$ except when ${\bf j} \in \mathcal{N}_i$. In this way, we get a dependency graph $(\mathcal{I}_n, \sim)$ for $(\xi_{{\bf i},n}, {\bf i} \in \mathcal{I}_n).$ 

By connectedness all vertices of any $\Gamma$-subgraph of $G(\mathcal{X}_n; p,r)$ lie within a distance $(k-1)r_n$ of one another, and hence, with $\theta$ denoting the volume of the unit ball in $\mathbb{R}^d$, we have

\begin{equation}\label{equ:poisson}
\begin{array}{rcl}
\displaystyle
E\xi_{{\bf i},n} 
\displaystyle
&\leq&
p^m\int_{\mathbb{R}^d}\cdots\int_{\mathbb{R}^d}h_{\Gamma,n}(\{x_1,...,x_k\})f(x_1)^kdx_k...dx_1\\
\displaystyle
&+&
p^m \int_{\mathbb{R}^d}\cdots\int_{\mathbb{R}^d}h_{\Gamma,n}(\{x_1,...,x_k\})\times\left(\prod_{i=1}^kf(x_i)-f(x_1)^k\right)\prod_{i=1}^kdx_i\\
\displaystyle
&+&
\sum_{j=m+1}^{\binom{k}{2}}\binom{j}{m}p^m(1-p)^{j-m}\int_{\mathbb{R}^d}\cdots\int_{\mathbb{R}^d}g_{\Gamma,n}(\{x_1,...,x_k\})f(x_1)^kdx_k...dx_1\\
\displaystyle
&+&
\sum_{j=m+1}^{\binom{k}{2}}\binom{j}{m}p^m(1-p)^{j-m} \int_{\mathbb{R}^d}\cdots\int_{\mathbb{R}^d}g_{\Gamma,n}(\{x_1,...,x_k\})\times\left(\prod_{i=1}^kf(x_i)-f(x_1)^k\right)\prod_{i=1}^kdx_i\\
\displaystyle
&\leq &
p^m\int_{B(x_1,kr_n)}\cdots\int_{B(x_1,kr_n)}f(x_1)^{k-1}dx_k...dx_2\int_{\mathbb{R}^d}h_{\Gamma,n}(\{x_1,...,x_k\})f(x_1)dx_1\\
\displaystyle
&+&
p^m \int_{\mathbb{R}^d}\cdots\int_{\mathbb{R}^d}h_{\Gamma,n}(\{x_1,...,x_k\})\times\left(\prod_{i=1}^kf(x_i)-f(x_1)^k\right)\prod_{i=1}^kdx_i\\
\displaystyle
&+&
p^m\frac{1-p}{p}\binom{\binom{k}{2}}{m}\int_{B(x_1,kr_n)}\cdots\int_{B(x_1,kr_n)}f(x_1)^{k-1}dx_k...dx_2\int_{\mathbb{R}^d}h_{\Gamma,n}(\{x_1,...,x_k\})f(x_1)dx_1\\
\displaystyle
&+&
p^m\frac{1-p}{p}\binom{\binom{k}{2}}{m} \int_{\mathbb{R}^d}\cdots\int_{\mathbb{R}^d}g_{\Gamma,n}(\{x_1,...,x_k\})\times\left(\prod_{i=1}^kf(x_i)-f(x_1)^k\right)\prod_{i=1}^kdx_i\\
\displaystyle
&\leq&
p^m(f_{\max}\theta(kr_n)^d)^{k-1}+p^{m-1}\binom{\binom{k}{2}}{m}(f_{\max}\theta(kr_n)^d)^{k-1}\\
\displaystyle
&\leq&
p^{m-1}(f_{\max}\theta(kr_n)^d)^{k-1}(1+C'),
\end{array}
\end{equation}
in which $C'=\binom{\binom{k}{2}}{m}$. Considering the third items on the right side of the first inequality, we sum up all the possibilities of the graphs which contains" strictly" subgraph $\Gamma$, and keep the $m$ edges we want and delete all other edges unwanted; and then we bound $\sum_{j=m+1}^{\binom{k}{2}}\binom{j}{m}(1-p)^{j-m}$ by $\binom{\binom{k}{2}}{m}\sum_{j=1}^{\infty}(1-p)^j$.

\smallskip
We can also get
$$card(\mathcal{N}_i)=\binom{n}{k}-\binom{n-k}{k}=k!^{-1}k^2n^{k-1}+O(n^{k-2}),$$
which leads to
$$\sum_{{\bf i}\in \mathcal{I}_n}\sum_{{\bf j} \in \mathcal{N}_i} E\xi_{{\bf i},n} E\xi_{{\bf j},n}\le c'p^{2(m-1)}n^{2k-1}r_n^{2d(k-1)}=c'p^{2(m-1)}n^{k+1}r_n^{dk}(nr_n^d)^{k-2}.$$

The next step is to bound $E{\xi_{{\bf i},n}}{\xi_{{\bf j},n}}$ when ${\bf i}\sim {\bf j}$ and ${\bf i}\neq {\bf j}.$
Here we have
$$h=\mid{\bf i}\cap {\bf j}\mid \in \{1,...,k-1\}.$$
By the same arguments as we bound $E\xi_{{\bf i},n}$, we can get 
\begin{equation}\label{equ:poisson1}
E[\xi_{{\bf i},n}\xi_{\xi_{\bf j},n}]\le C'' p^{2m-h+1}(f_{\max}\theta(2kr_n)^d)^{2k-h-1}.
\end{equation}

Given $h\in\{1,2,...,k-1\}$, the number of pairs $({\bf i},{\bf j})\in \mathcal{I}_n\times \mathcal{I}_n$ with $h$ elements in common is 
$$\binom{n}{k}\binom{k}{h}\binom{n-k}{k-h}=\Theta(n^{2k-h}).$$

Finally, we get
\begin{equation}\label{equ:poisson2}
\begin{array}{rcl}
\displaystyle
\sum_{{\bf i}\in \mathcal{I}_n}\sum_{{\bf j} \in {\mathcal{N}_i}\setminus \{i\}} E\xi_{{\bf i},n} \xi_{{\bf j},n}
\displaystyle
&\le &
C\sum_{h=1}^{k-1}p^{2m-h+1}n^{2k-h}r_n^{d(2k-h-1)}\\
\displaystyle
&=&
C\sum_{h=1}^{k-1}p^{2m-h+1}n^{k+1}r_n^{dk}(nr_n^d)^{k-h-1}\\
\displaystyle
&\le&
Cp^{2m-k+2}n^{k+1}r_n^{dk}\sum_{h=1}^{k-1}(nr_n^d)^{k-h-1}\\
\displaystyle
&=&
c''p^{2m-k+2}n^{k+1}r_n^{dk},
\end{array}
\end{equation}
the last equality holds as the condition: $(nr_n^d)_{n\ge 1}$ is a bounded sequence.

\smallskip
From the bound (~\ref{equ:poisson1}) and (~\ref{equ:poisson2}) and Theorem~\ref{thm:stein}, we have

\begin{displaymath}\label{equ:poisson}
d_{TV}(G'_n, Z_n)=\left\{\begin{array}{ll}  cnp_n^{2m+2-k}r^d_n  &\textrm{if $k\ge 4$}\\
cnp_n^{2m-2}r_n^d &\textrm{if $2\le k<4$}
\end{array}\right.
\end{displaymath}

From Theorem ~\ref{thm:similar}, we can get the Poisson approximation; and the fact that the convergence of the standardized Po$(\lambda)$ to the normal when $\lambda \to \infty$ gives us the remaining assertion of the theorem.
\end{proof}

\subsection{Proof of Theorem~\ref{thm:betti}}
\begin{proof}
From the definition of percolated random Vietoris-Rips complex, if a simplex $\sigma$ exists, which means each of its 2-faces exists, i.e., this underlying subgraph with vertex set same as of $\sigma$ is a complete graph.  And the 1-skeleton of the cross polytope $\mathcal{O}_k$ has $2^2\binom{k}{2}=2k(k-1)$ 1-faces. Then with the same arguments as in the proof of Theorem~\ref{thm:kahle} in~\cite{kahle2011random}, we can get $$E[\tilde{o}'_k]=\Theta(n^{2k+2}r^{(2k+1)d}p^{2k(k-1)}).$$
Then slightly modify the proof of Theorem~\ref{thm:kahle} in~\cite{kahle2011random}, we finish the proof.
\end{proof}
\section{Possible generalization}
One of the difficulties on the counting of \emph{induced} subgraphs of random geometric graphs arises from the complicated geometric structures. In the \emph{percolated random geometric graphs} $G(\mathcal{X}_n; p, r)$, there are more $\Gamma$-subgraphs than we can get directly from $G(\mathcal{X}_n; r)$ by keeping edges. In other words, there exist many subgraphs which are \emph{feasible} in $G(\mathcal{X}_n; p, r)$ but not in $G(\mathcal{X}_n; r),$ however, if some subgraph $\Gamma_0$ is not feasible for $G(\mathcal{X}_n; p, r)$, of course, it is not feasible for $G(\mathcal{X}_n; r),$ either.
We disturb the geometric structures of $G(\mathcal{X}_n; r)$ by deleting the unwanted edges with certain probability $1-p$. Roughly speaking, there are "fewer" edges in $G(\mathcal{X}_n; p, r)$, but "more" \emph{induced} subgraphs with some positive probability. 

\smallskip
In this paper, we only explore the counting of \emph{induced} subgraphs on \emph{percolated random geometric graph}, which is the simplest\emph{ soft random geometric graph}. Can we extend the counting to the general\emph{ soft random geometric graph} with other  more general connection functions(e.g., see \cite{penrose2016connectivity})? We would be very interested to see more results related to this topic.

Moreover, as we mentioned already, random geometric simplicial complexes is extensively studied in these years. There are of course a lot of interesting and challenging open problems in this areas, see the last part of ~\cite{bobrowski2014topology} for some of them. We would like to explore more in further directions.
 \bibliographystyle{plain}
 \bibliography{counting}

 \bigskip
 {\footnotesize \obeylines \parindent=0pt
Anshui Li 
Department of Mathematics, Hangzhou Normal University, Hangzhou, Zhejiang, P.R.China, 310000. 
E-mails: \texttt{anshuili@hznu.edu.cn}.
}

 \end{document}